\documentclass{amsart}

\usepackage{amssymb,latexsym, amscd, amsthm, amsmath}
\input xypic
\xyoption{all}
\newcommand{\comments}[1]{}

\theoremstyle{plain}
\newtheorem{theorem}{Theorem}[section]

\newtheorem{lemma}[theorem]{Lemma}

\newtheorem{proposition}[theorem]{Proposition}
\newtheorem{definition}[theorem]{Definition}

\newtheorem{corollary}[theorem]{Corollary}

\newtheorem{remark}[theorem]{Remark}
\newtheorem{conjecture}[theorem]{Conjecture}

\DeclareSymbolFont{AMSb}{U}{msb}{m}{n}
\DeclareMathSymbol{\F}{\mathbin}{AMSb}{"46}
\DeclareMathSymbol{\N}{\mathbin}{AMSb}{"4E}
\DeclareMathSymbol{\Z}{\mathbin}{AMSb}{"5A}
\DeclareMathSymbol{\R}{\mathbin}{AMSb}{"52}
\DeclareMathSymbol{\Q}{\mathbin}{AMSb}{"51}
\DeclareMathSymbol{\I}{\mathbin}{AMSb}{"49}
\DeclareMathSymbol{\C}{\mathbin}{AMSb}{"43}
\DeclareMathSymbol{\G}{\mathbin}{AMSb}{"47}
\DeclareMathAlphabet{\mathscr}{OT1}{pzc}{m}{it}

\newcommand{\ra}{\rightarrow}
\newcommand{\xra}[1]{\xrightarrow{#1}}
\newcommand{\ol}[1]{\overline{#1}}
\newcommand{\brkt}[1]{\langle #1 \rangle}

\newcommand{\isom}{\simeq}
\newcommand{\tensor}{\otimes}

\newcommand{\vp}{\varphi}
\newcommand{\mc}[1]{\mathcal{#1}}

\begin{document}

\title[The refined Coates--Sinnott conjecture]{The Refined Coates--Sinnott Conjecture for Characteristic $p$ Global Fields}

\author{Joel Dodge}
\address{Binghamton University, Department of Mathematical Sciences, Binghamton, NY 13902-6000}
\curraddr{}
\email{joel@math.binghamton.edu}
\thanks{}

\author{Cristian D. Popescu*}
\address{University of California, San Diego, Department of Mathematics, La Jolla, CA 92093-0112.}
\curraddr{}
\email{cpopescu@math.ucsd.edu}
\thanks{${}^*$The second author's research was supported by NSF Grant DMS-0901447.}



\date{}

\dedicatory{}

\begin{abstract}
This article is concerned with proving a refined function field analogue of the Coates-Sinnott conjecture, formulated in the number field context in 1974.  Our main theorem calculates the Fitting ideal of a certain even Quillen $K$-group in terms of special values of $L$-functions.  The techniques employed are directly inspired by recent work of Greither and Popescu in the equivariant Iwasawa theory of arbitrary global fields.  They rest on the results of Greither--Popescu on the Galois module structure of certain naturally defined Picard $1$-motives associated to an arbitrary Galois extension of function fields.
\end{abstract}

\maketitle

\section{Introduction}

In \cite{GP1}, Greither and the second author engage in an in-depth study of the Galois module structure of Picard $1$-motives associated to characteristic $p$ global fields.  As an application of their results, they prove the $\ell$-adic \'{e}tale cohomological Coates-Sinnott conjecture in function fields, for all primes $\ell\neq p$.  In the main result of this article we strengthen their result in two distinct directions.

For primes $\ell\neq p$, inspired by techniques used in \cite{GP3} to prove a refined Coates-Sinnott conjecture in number fields, we strengthen the results in \cite{GP1} on the $\ell$-adic \'{e}tale cohomological Coates-Sinnott conjecture in function fields by proving an equality of ideals instead of just a containment.  For $\ell=p$, we prove that the relevant special $L$--value is a unit in the group ring with $\Z_p$-coefficients.  Finally, using the Quillen-Lichtenbaum conjecture (now a theorem), we transfer these results into statements about Quillen $K$-groups and combine them to give an especially strong form of the full integral Coates-Sinnott conjecture in the function field context.

We begin by giving the statement of the classical Coates--Sinnott conjecture for number fields.  If $K/k$ is an abelian extension of number fields with Galois group $G$ and $S$ is a finite set of primes of $k$ containing the infinite primes, then let
$$\Theta_{K/k, S}(s) = \sum_{\chi\in\widehat G} L_S(\chi,s) e_{\chi^{-1}}: \C\to\C[G]$$
be the $\C[G]$-valued, $S$-incomplete, $G$-equivariant $L$-function.  Here, $\widehat G$ denotes the group of complex valued characters of $G$, $e_\chi$ is the usual idempotent associated to $\chi$ in $\C[G]$ and $L_S(\chi, s)$ is the global $L$--function associated to $\chi$ with Euler factors at all primes in $S$ removed. The original Coates-Sinnott conjecture was stated in \cite{Coates-Sinnott} in terms of the Quillen $K$--groups $K_i(O_K)$ of the ring of algebraic integers $\mc{O}_K$.

\begin{conjecture}[Coates-Sinnott]\label{original coates-sinnott}
Let $K/k$ be an abelian extension of number fields with Galois group $G$, let $S$ be a finite set of primes of $k$ containing the infinite primes and the primes which ramify in $K/k$ and let $n\geq 2$ be an integer.  Then
$$\mathrm{Ann}_{\Z[G]}(K_{2n-1}(\mc{O}_{K})_{\mathrm{tors}}) \cdot \Theta_{K/k, S}(1-n) \subseteq \mathrm{Ann}_{\Z[G]}(K_{2n-2}(\mc{O}_K)).$$
\end{conjecture}

\begin{remark}

\begin{enumerate}
\item[]
\item It is known from work of Deligne-Ribet \cite{Deligne-Ribet} that $$\mathrm{Ann}_{\Z[G]}(K_{2n-1}(\mc{O}_K)_{\mathrm{tors}}) \cdot \Theta_{K/k, S}(1-n) \subseteq \Z[G].$$ (See \cite{CS} for more details.)

\item We call the statement in Conjecture \ref{original coates-sinnott} {\em unrefined} because it only predicts a containment and not an equality of ideals.  In order to prove an equality, we will need a statement involving invariants which are more subtle than the annihilator, namely the first Fitting ideal ${\rm Fit}_{\Z[G]}(\ast)$.

\end{enumerate}
\end{remark}

Many of these results are most naturally stated and proven in terms of \'{e}tale cohomology rather than Quillen $K$--theory.  We state here the Quillen-Lichtenbaum conjecture linking \'etale cohomology and Quillen $K$--theory, recently proven as a consequence of work of Voevodsky \cite{Voevodsky} on the Bloch-Kato conjecture.   This will be repeatedly used to translate between results on \'{e}tale cohomology and $K$-theory.

\begin{theorem}[Quillen-Lichtenbaum Conjecture]\label{QL}
Let $K/k$ be an abelian extension of global fields with Galois group $G$ and let $S$ be a finite, nonempty set of primes of $k$, containing all the archimedean (infinite) primes.  Let $\ell$ be a prime number which is assumed to be odd if ${\rm char}(k)=0$ and different from $p$ if ${\rm char}(k)=p>0$. Then, for all $n\geq 2$ and $i=1, 2$, Soul\'e's $\ell$--adic Chern character maps \cite{Soule} give isomorphisms of $\Z_\ell[G]$-modules
$$K_{2n-i}(\mc{O}_{K, S})\tensor \Z_\ell \isom H^i_{\acute{e}t}(\mc{O}_{K, S}, \Z_\ell(n)),$$
where $\mc{O}_{K, S}$ is the usual ring of $S$--integers in $K$.
\end{theorem}

Away from the prime $\ell=2$ and under the assumption that a certain Iwasawa $\mu$--invariant vanishes, an especially strong form of Conjecture \ref{original coates-sinnott} was proved by Greither-Popescu in \cite{GP3}, as a consequence of their equivariant main conjecture in Iwasawa theory, also proved in \cite{GP3}.   Namely, they prove the following (see Theorem 6.11 in \cite{GP3}).

\begin{theorem}[Greither-Popescu]\label{refined coates-sinnott number fields}
Keep the notations and assumptions of Conjecture \ref{original coates-sinnott}.  Let $\ell>2$ and suppose that the Iwasawa $\mu$-invariant for the cyclotomic $\Z_\ell$--extension of $K$ is zero.  Then, we have an equality
\begin{eqnarray}
\nonumber \mathrm{Ann}_{\Z_\ell[G]}(H^1_{\acute{e}t}(\mc{O}_{K, S}[1/\ell], \Z_\ell(n))_{\mathrm{tors}}) \cdot \Theta_{K/k, S}(1-n) = \\
\nonumber = \epsilon \cdot \mathrm{Fit}_{\Z_\ell[G]}(H^2_{\acute{e}t}(\mc{O}_{K, S}[1/\ell], \Z_\ell(n))),
\end{eqnarray}
for a certain idempotent $\epsilon$ in $\Z_\ell[G]$.

\end{theorem}

\begin{remark}

 \begin{enumerate}

\item[]

\item If $R$ is a commutative ring and $M$ is a finitely generated $R$-module, then we write $\mathrm{Fit}_R(M)$ for the first Fitting ideal of $M$.

\item For the definition of $\epsilon$, see the remarks following Corollary 6.10 in \cite{GP3}.

\item One can use Theorem \ref{QL} and the fact that the Fitting ideal ${\rm Fit}_{\Z_\ell[G]}(M)$ of a finitely generated $\Z_\ell[G]$--module $M$ is included in its annihilator  ${\rm Ann}_{\Z_\ell[G]}(M)$
to show that Theorem \ref{refined coates-sinnott number fields} is a refinement of the $\ell$--primary piece of Conjecture \ref{original coates-sinnott}. (See \cite{GP3}, \S6 for the relevant details.)

\item The theorem above is understood as a {\em refined} result because it expresses the left hand side as a precise piece of the Fitting ideal of an algebraic invariant associated to the field $K$.
\end{enumerate}
\end{remark}

In the case of characteristic $p$ global fields, an $\ell$-adic \'{e}tale cohomological version of Conjecture \ref{original coates-sinnott} at primes $\ell\ne p$ has recently been proven by Greither-Popescu in \cite{GP1}.  In fact, they prove a stronger, though still unrefined, statement.  We briefly describe their result below.

Let $\mc{K}_0/\mc{K}_0'$ be an abelian extension of characteristic $p$ global fields with Galois group $G$ and let $S_0$ be a finite nonempty set of primes of $\mc{K}_0'$ containing all those primes which ramify in $\mc{K}_0/\mc{K}_0'$.  In what follows, $\mathcal O_{\mc{K}_0, S_0}$ denotes the ring of $S_0$-integers in $\mc{K}_0$. The $S_0$-incomplete $G$-equivariant $L$-function for this extension will be denoted by
$$\Theta_{\mc{K}_0/\mc{K}_0', S_0}:\C \ra \C[G].$$
The precise definition of $\Theta_{\mc{K}_0/\mc{K}_0', S_0}$ will be given in \S\ref{L-functions} below.
Let $q$ be the order of the exact field of constants of $\mc{K}_0'$, i.e. the finite field consisting of those elements of $\mc{K}_0'$ which are algebraic over $\F_p$.  Theorem 5.20 in \cite{GP1} is the following.

\begin{theorem}[Greither-Popescu]\label{unrefined CS in function fields}
Let $\ell$ be a prime different from $p$.  Then
$$\mathrm{Ann}_{\Z_\ell[G]}(H^1_{\acute{e}t}(\mc{O}_{\mc{K}_0, S_0}, \Z_\ell(n))) \cdot \Theta_{\mc{K}_0/\mc{K}_0', S_0}(q^{n-1}) \subseteq \mathrm{Fit}_{\Z_\ell[G]}(H^2_{\acute{e}t}(\mc{O}_{\mc{K}_0, S_0}, \Z_\ell(n))).$$
\end{theorem}

\begin{remark}
\begin{enumerate}
\item[]

\item The values $\Theta_{\mc{K}_0/\mc{K}_0', S_0}(q^{n-1})$ are the correct function field analogues of the special values at $s=1-n$ in the number field case.

\item We do not have to take the torsion subgroup on the left hand side because these \'{e}tale cohomology groups are finite in function fields (as a direct consequence of the characteristic $p$ Riemann hypothesis, see \cite{GP1} for more details.)
\end{enumerate}
\end{remark}

The fact that an equality was obtained in Theorem \ref{refined coates-sinnott number fields} makes it reasonable to expect that the result of Theorem \ref{unrefined CS in function fields} can be strengthened to give an equality in the function field case as well.  As we prove in this paper, this is indeed the case at primes $\ell\neq p$. Moreover, the passage to Quillen $K$--theory permits us to prove the analogous equality at the prime $\ell=p$ as well.
First, we obtain the following result in terms of \'etale cohomology. (See \S5 below for the proof.)

\begin{theorem} \label{main theorem}
Let $n\geq 2$ be an integer.  Then, the following hold.
\begin{enumerate}
\item $\mathrm{Fit}_{\Z_\ell[G]}\, H^1_{\acute{e}t}(\mc{O}_{\mc{K}_0, S_0}, \Z_\ell(n)) \cdot \Theta_{\mc{K}_0/\mc{K}_0', S_0}(q^{n-1})
= \mathrm{Fit}_{\Z_\ell[G]}\, H^2_{\acute{e}t}(\mc{O}_{\mc{K}_0, S_0}, \Z_\ell(n)).$
\item $\Theta_{\mc{K}_0/\mc{K}_0', S_0}(q^{n-1}) \in \Z_p[G]^\times.$
\end{enumerate}
\end{theorem}

\begin{remark}
\begin{enumerate}
\item[]
\item The previous hypothesis on the $\mu$-invariant is not present in the Theorem above because the characteristic $p$ analogue of the $\mu$-invariant is always zero. (That is, for all prime $\ell$, the $\ell$--adic Tate module of the Jacobian of a smooth projective model of $\mathcal K_0$ is a free $\Z_\ell$--module.)

\item No idempotent appears here because, unlike in the case of number fields, the special values $\Theta_{\mc{K}_0/\mc{K}_0', S_0}(q^{n-1})$ are supported at all idempotents of $\Q_\ell[G]$.
This is a consequence of the fact that characteristic $p$ global $L$--functions do not vanish at negative integers.
\item Although the containment of Theorem \ref{unrefined CS in function fields} can be viewed as giving one half of Theorem \ref{main theorem}, our proof of Theorem \ref{main theorem} does not make use of that result. However, we make
use in an essential way of the main result of \cite{GP1}, which can be viewed as an abelian equivariant main conjecture in the Iwasawa theory of characteristic $p$ global fields.
\end{enumerate}
\end{remark}

From Theorem \ref{main theorem}, we can quickly deduce our main theorem.

\begin{theorem}\label{refined CS}
Let $n\geq 2$ be an integer.  Then
$$\mathrm{Fit}_{\Z[G]}(K_{2n-1}(\mc{O}_{\mc{K}_0, S_0})) \cdot \Theta_{\mc{K}_0/\mc{K}_0', S_0}(q^{n-1}) = \mathrm{Fit}_{\Z[G]}(K_{2n-2}(\mc{O}_{\mc{K}_0, S_0})).$$
\end{theorem}

\begin{proof}
It will suffice to prove that we have equality of ideals in $\Z_\ell[G]$ after tensoring both sides with $\Z_\ell$, for each prime $\ell$.

If $\ell \neq p$, then the equality follows immediately from Theorem \ref{QL} combined with  part $(1)$ of Theorem \ref{main theorem} and the fact that Fitting ideals commute with base change.

If $\ell=p$, then the equality follows from part $(2)$ of the previous theorem and the fact that $K_{n}(\mc{O}_{\mc{K}_0, S_0}) \tensor \Z_p = 0$, for all $n>0$.  To see this we make a couple of observations.  First, Theorem 1.3 in \cite{Kolster} (a well known result of Geisser--Levine) and the $5$--term exact sequence which follows Theorem 1.5 in loc.cit., imply that $K_n(\mc{O}_{\mc{K}_0, S_0})$ has no $p$-torsion, for all $n>0$.  Fix an $n>0$. If $K_n(\mc{O}_{\mc{K}_0, S_0})$ had a free $\Z$--submodule, then $K_n(\mc{O}_{\mc{K}_0, S_0})\tensor \Z_\ell$ would not be finite for {\em any} $\ell$.  It is known that the groups $H^i_{\acute{e}t}(\mc{O}_{\mc{K}_0, S_0}, \Z_\ell(n))$ are finite for all $\ell \neq p$ and $i=1, 2$ and a final appeal to the Quillen-Lichtenbaum conjecture finishes the proof.
\end{proof}

\section{Algebraic Background}

In this section we introduce the relevant group rings and a number of functors on modules over these rings.  We then discuss some of the interplay between these functors and Fitting ideals.  The first Fitting ideal, from now on just referred to as {\em the Fitting ideal,} plays a key role in formulating refined versions of the classical conjectures on special values of $L$-functions. Most of the material in this section can be found in the literature and we will give references rather than complete proofs whenever possible.

\subsection{Some Functors on Modules over Group
Rings}\label{functors section}
Let $\mc{K}_0/\mc{K}_0'$ be a
finite abelian extension of characteristic $p$ global fields with
Galois group $G$.  We remind the reader that $\mc{K}_0$ and
$\mc{K}_0'$ are finite extensions of $\Bbb F_p(t)$, for some
variable $t$.

Throughout, all fields will be viewed as subfields of a fixed
separable closure of $\mc{K}_0$. Let $\kappa$ be an algebraic
closure of $\F_p$ and let $\mc{K} = \mc{K}_0\kappa$, $\mc{K}' =
\mc{K}_0' \kappa$ be field composita of $\kappa$ with $\mc{K}_0$
and $\mc{K}_0'$.  Let $\mc{G} := G(\mc{K}/\mc{K}_0')$. Since we
have an isomorphism of profinite groups $G(\mc{K}/\mc{K}_0)\simeq
\widehat\Z$ (where $\widehat \Z$ is the profinite completion of
$\Z$), the profinite group $\mc{G}$ is an extension of $G$ by
$\widehat \Z$.

Let $\ell$ be a prime.  A basic fact about the group ring
$\Z_\ell[G]$ is the following.  If $G'$ is the $\ell$--Sylow
subgroup of $G$ and $G \isom \Delta\times G'$, then we have a ring
isomorphism
\begin{equation}\label{decomposition}\Z_\ell[G] \isom \bigoplus _{\widetilde{\chi}} \Z_\ell[\chi][G'],\end{equation}
where the sum ranges over $G(\ol{\Q_\ell}/\Q_\ell)$--conjugacy classes of characters $\chi \in \widehat{\Delta}(\ol{\Q_\ell})$ (the group of $\ol{\Q_\ell}$--valued characters of $\Delta$)  and $\Z_\ell[\chi]$ is the ring generated over $\Z_\ell$ by the values of $\chi$.  Each $\Z_\ell[\chi][G']$ is a local ring with maximal ideal $\brkt{\ell, I_{G'}}$, where $I_{G'}$ is the augmentation ideal of $\Z_\ell[\chi][G']$.

Let $\Z_\ell[[\mc{G}]]$ be the $\ell$--adic profinite group ring
associated to $\mc{G}$, endowed with the usual $\ell$--profinite
topology. By a $\Z_\ell[[\mc{G}]]$--module we always mean a
finitely generated topological $\Z_\ell[[\mc{G}]]$--module. We
will record the definitions and some basic properties of several
functors defined on the category of $\Z_\ell[[\mc{G}]]$--modules.

Let $\Gamma$ be an open subgroup of $\mc{G}$ which is topologically cyclic with generator $\gamma$ (e.g. $\Gamma:=G(\mathcal K/\mathcal K_0)$.) The invariants and coinvariants functors associated to $\Gamma$ are
$$M^\Gamma = \{m\in M \textnormal{ $|$ } \gamma\cdot m = m\}, \qquad M_\Gamma = M/ \{m-\gamma \cdot m\textnormal{ $|$ } m \in M\},$$
respectively, for all $\Z_\ell[[\mathcal G]]$--modules $M$. Both $M^\Gamma$ and $M_\Gamma$ are naturally $\Z_\ell[\mc{G}/\Gamma]$-modules.  In particular, if $\Gamma = G(\mc{K}/\mc{K}_0)$ then both $M^\Gamma$ and $M_\Gamma$ are $\Z_\ell[G]$-modules.

\begin{lemma}\label{snake}
Suppose that $0\ra A\ra B\ra C\ra 0$ is a short exact sequence of finitely generated $\Z_\ell[[\mc{G}]]$-modules.  Then there is an exact sequence of $\Z_\ell[\mc{G}/\Gamma]$-modules
$$0 \ra A^\Gamma \ra B^\Gamma \ra C^\Gamma \ra A_\Gamma \ra B_\Gamma \ra C_\Gamma \ra 0.$$
\end{lemma}

\begin{proof}
This follows from the snake lemma combined with the exactness of $$0 \ra M^\Gamma \ra M \xra{1-\gamma} M \ra M_\Gamma \ra 0$$ for any $\Z_\ell[[\mc{G}]]$-module $M$.
\end{proof}

\begin{lemma}\label{finiteness of invariants}
Let $M$ be a $Z_\ell[[\mc{G}]]$--module which is finitely
generated over $\Z_\ell$, then $M_\Gamma$ is finite if and only if
$M^\Gamma$ is finite.
\end{lemma}

\begin{proof}
As $M$ is finitely generated over $\Z_\ell$, $M^\Gamma$ is finite if and only if $M^\Gamma \tensor_{\Z_\ell}\Q_\ell = 0$ and similarly for $M_\Gamma$. The alternating sum of the $\Q_\ell$-vector space dimensions in
$$0 \ra M^\Gamma \tensor_{\Z_\ell} \Q_\ell \ra M\tensor_{\Z_\ell}\Q_\ell \xra{(1-\gamma)\tensor 1} M\tensor_{\Z_\ell}\Q_\ell \ra M_\Gamma \tensor_{\Z_\ell}\Q_\ell \ra 0$$ is $0$ and so we have that $\mathrm{dim}_{\Q_\ell}(M_\Gamma \tensor_{\Z_\ell}\Q_\ell) = \mathrm{dim}_{\Q_\ell}(M^\Gamma \tensor_{\Z_\ell}\Q_\ell)$.  The lemma follows.
\end{proof}

\begin{definition} If $M$ is a $\Z_\ell[[\mc{G}]]$--module, then define

\begin{enumerate}

\item $M^* = \mathrm{Hom}_{\Z_\ell}(M, \Z_\ell)$, with
    $\mc{G}$--action $(g\cdot \vp)(m) = \vp(g^{-1}\cdot m)$;

\item $M^\vee = \mathrm{Hom}_{\Z_\ell}(M, \Q_\ell/\Z_\ell)$,
    with $\mc{G}$--action $(g\cdot\vp)(m) = \vp(g^{-1}\cdot
    m)$;

\item $M^\wedge = \mathrm{Hom}_{\Z_\ell}(M, \Q_\ell/\Z_\ell)$,  with $\mc{G}$--action $(g\cdot\vp)(m) = \vp(g\cdot m)$.
\end{enumerate}
By continuity, each of these $\mc{G}$--actions extends naturally and uniquely
to an action of $\Z_\ell[[\mc{G}]]$.
\end{definition}

\begin{lemma} \label{duality}
Suppose that $M$ is a finitely generated $\Z_\ell[[\mc{G}]]$-module which is $\Z_\ell$-free and that $M_\Gamma$ is finite.  Then there is an isomorphism of $\Z_\ell[[\mc{G}]]$-modules:
$$(M_\Gamma)^\vee \isom (M^*)_\Gamma$$
\end{lemma}

\begin{proof}
See Lemma 5.16 in \cite{GP1}.
\end{proof}

Let $\mu_{\ell^\infty}$ denote the group of $\ell$-power roots of
unity in $\mc{K}$.  Since $\mu_{\ell^\infty}\subseteq \mc{K}$, we
have the $\ell$-cyclotomic character $c_\ell:\mc{G} \ra
\Z_\ell^\times$ which is characterized by $g\cdot \zeta =
\zeta^{c_\ell(g)}$, for all $\zeta \in \mu_{\ell^\infty}$ and all
$g\in \mc{G}$. This allows us to define a family of continuous
$\Z_\ell$-algebra automorphisms $t_n: \Z_\ell[[\mc{G}]] \ra
\Z_\ell[[\mc{G}]]$ which are uniquely characterized by $t_n(g) =
c_\ell(g)^n g$, for all $n\in\Z$ and all $g \in \mc{G}$.  It easy
to check that $(t_n)^{-1} = t_{-n}$.  If $M$ is a
$\Z_\ell[[\mc{G}]]$-module, then the $n$--th Tate twist $M(n)$ of
$M$ is the module obtained by extending scalars along the morphism
$t_{-n}$. The $\mc{G}$--action on $M(n)$ is given by
$$g*m = c_\ell(g)^n g\cdot m.$$
It is straightforward to check that one has canonical $\Z_\ell[[\mathcal G]]$--module isomorphisms
$$M(n)\simeq M\otimes_{\Z_\ell}\Z_\ell(n), \qquad M(n)(m)\simeq M(n+m),$$
for all $n, m \in \Z$. Clearly, the functor $M\to M(n)$ is exact for all $n$.

\begin{remark}\label{eigenvalues of twists}
If $\gamma$ is a topological generator of $\Gamma$ which restricts
to the $q^\alpha$-power Frobenius on $\kappa$, then we have
$c_\ell(\gamma) = q^\alpha$.  A simple calculation shows that if
$V$ is a $\Q_\ell$--vector space on which $\gamma$ acts with
eigenvalue $\lambda$, then $\gamma$ acts on $V(n)$ with eigenvalue
$q^{n\alpha}\lambda$.
\end{remark}

\begin{remark}\label{twists of duals}
We can immediately verify the isomorphisms of $\Z_\ell[[\mc{G}]]$-modules:
$$M^*(n) \isom M(-n)^*\textrm{, and }M^\vee(n)\simeq M(-n)^\vee.$$\
\end{remark}

\subsection{Fitting Ideals}
Let $R$ be a commutative ring.  We refer the reader to the
Appendix in \cite{Mazur-Wiles} for the definition of the Fitting
ideal $\mathrm{Fit}_R(M)$ of a finitely generated $R$-module $M$.
Here we record some basic facts about Fitting ideals over
arbitrary rings and then present some properties that are special
to modules over the group rings $\Z_\ell[G]$.  The notion of
projective dimension plays an important role in the calculation of
Fitting ideals. Throughout the rest of this paper
$\mathrm{pd}_R(M)$ will mean the projective dimension of the $R$--module $M$.

\begin{proposition}\label{properties of fitting ideals}
Let $M, N$ be finitely generated $R$-modules, then

\begin{enumerate}

\item $\mathrm{Fit}_R(M) \subseteq \mathrm{Ann}_R(M)$,

\item $\mathrm{Fit}_R(R/I) = \mathrm{Ann}_R(R/I) = I$,

\item $\mathrm{Fit}_R(M\oplus N) = \mathrm{Fit}_R(M)\cdot \mathrm{Fit}_R(N)$.

\item If $R\xra{\pi} S$ is a morphism of rings, then $\pi(\mathrm{Fit}_R(M))\cdot S = \mathrm{Fit}_S(M\tensor_R S)$.
\end{enumerate}
\end{proposition}

\begin{proof}
See the Appendix in \cite{Mazur-Wiles} and the references therein.
\end{proof}

\begin{proposition}\label{linear algebra}
Let $M$ be a $\Z_\ell[[\mc{G}]]$-module and let $m\in \Z$.  Then
$$\mathrm{Fit}_{\Z_\ell[[\mc{G}]]}(M(m)) = t_{-m}(\mathrm{Fit}_{\Z_\ell[[\mc{G}]]}(M)).$$
\end{proposition}

\begin{proof}
See Lemma 3.1 in \cite{CS}.
\end{proof}

In the next proposition we need a particular $\Z_\ell$-algebra involution of $\Z_\ell[G]$.  Define $\iota: G \ra G$ by $\iota(g) = g^{-1}$ and extend to $\Z_\ell[G]$ by $\Z_\ell$-linearity.

\begin{proposition}\label{projective dimension argument}
Let $G$ be a finite abelian group and suppose that $M$ is a finite $\Z_\ell[G]$-module with $\mathrm{pd}_{\Z_\ell[G]}(M) = 1$.  Then $\mathrm{pd}_{\Z_\ell[G]}(M^\vee)=1$ and $$\mathrm{Fit}_{\Z_\ell[G]}(M^\vee) = \iota(\mathrm{Fit}_{\Z_\ell[G]}(M)).$$  Furthermore, this Fitting ideal is principal, generated by an element which is not a zero-divisor in $\Z_\ell[G]$.
\end{proposition}

\begin{proof}
Lemma $6$ in \cite{Burns-Greither} deals with $M^\wedge$ instead
of $M^\vee$ but the proof in loc.cit. can easily be adapted to our
situation. Once we observe that, in the notation of
\cite{Burns-Greither}, $\alpha^T$ must be replaced by
$\iota(\alpha^T)$, the equality $\mathrm{Fit}_{\Z_\ell[G]}(M^\vee)
= \iota(\mathrm{Fit}_{\Z_\ell[G]}(M))$ follows from their
argument. In addition, it follows from the proof in loc.cit. that
$\mathrm{pd}_{\Z_\ell[G]}(M^\vee)=1$ and then Lemma 2.1 in
\cite{CS} implies that $\mathrm{Fit}_{\Z_\ell[G]}(M^\vee)$ is
principal and generated by a non zero-divisor in $\Z_\ell[G]$.
\end{proof}

\begin{lemma}\label{fitting and annhilators}
Let $M$ be a $\Z_\ell[G]$-module which is cyclic as an abelian group.  Then $$\mathrm{Fit}_{\Z_\ell[G]}(M) =\mathrm{Ann}_{\Z_\ell[G]}(M) = \mathrm{Ann}_{\Z_\ell[G]}(M^\wedge) = \mathrm{Fit}_{\Z_\ell[G]}(M^\wedge).$$
\end{lemma}

\begin{proof}
The two outer equalities are contained in Proposition \ref{properties of fitting ideals} so it will suffice to prove that $$\mathrm{Ann}_{\Z_\ell[G]}(M) = \mathrm{Ann}_{\Z_\ell[G]}(M^\wedge).$$
Due to the obvious $\Z_\ell[G]$-module isomorphism $(M^\wedge)^\wedge \isom M$ it will suffice to prove that $\mathrm{Ann}_{\Z_\ell[G]}(M) \subseteq \mathrm{Ann}_{\Z_\ell[G]}(M^\wedge)$.  For this, suppose that $x\cdot M =0$ for some $x\in \Z_\ell[G]$ and let $f\in M^\wedge$.  Then $(x\cdot f)(m) = f(x\cdot m) =f(0) =0$, for all $m\in M$.  This implies the containment and concludes the proof.
\end{proof}

The next property of Fitting ideals plays a key role in the upcoming calculations.

\begin{proposition}\label{four term lemma}
Suppose that
$$0\ra A\ra B\ra C \ra D \ra 0$$
is an exact sequence of $\Z_\ell[G]$-modules which are all finite and which also satisfy $\mathrm{pd}_{\Z_\ell[G]}(B)\leq 1$ and $\mathrm{pd}_{\Z_\ell[G]}(C) \leq 1$.  Then we have
 $$\mathrm{Fit}_{\Z_\ell[G]}(A^{\wedge})\cdot \mathrm{Fit}_{\Z_\ell[G]}(C) = \mathrm{Fit}_{\Z_\ell[G]}(B)\cdot \mathrm{Fit}_{\Z_\ell[G]}(D).$$
\end{proposition}
\begin{proof}
See Lemma 5 in \cite{Burns-Greither}.
\end{proof}
It is this Proposition which allows us to perform a precise
calculation of the Fitting ideal instead of simply arriving at a
containment as in the unrefined Coates-Sinnott conjecture.

\section{Jacobians, \'{e}tale cohomology groups and $L$-functions}

In this section we introduce the algebraic and analytic objects which appear in the statement of the refined Coates-Sinnott conjecture.  On the algebraic side, we introduce Jacobians of curves and \'{e}tale cohomology groups and present the relevant connections between the two.  On the analytic side, the equivariant $L$-functions, both at the finite and at the infinite level, are introduced.  Again, most of the material in this  section can be found in \cite{GP1}. The notations are those used in the previous section.

Let $Z_0'$ be a smooth projective model for $\mc{K}_0'$ over $\kappa_0$ and $Z_0$ be a smooth projective model for  $\mc{K}_0$ over $\kappa_0$.  That is, $Z_0$ and $Z_0'$ are smooth projective curves over $\kappa_0$ whose fields of rational functions are isomorphic to $\mc{K}_0$ and $\mc{K}_0'$.  Similarly $Z$ and $Z'$ will denote smooth projective models for $\mc{K}$ and $\mc{K}'$ over $\kappa$.

  Let $\kappa_0 := \kappa \cap
\mc{K}_0'$ be the exact field of constants of $\mc{K}'_0$ and let $q:=|\kappa_0|$. Then $\gamma_q$ will denote the
$q$-power Frobenius map, viewed as a distinguished topological
generator of $G(\kappa/\kappa_0)$.

Let $v$ be a prime of $\mc{K}_0'$ which is unramified in
$\mc{K}/\mc{K}_0'$ (equivalently, unramified in
$\mc{K}_0/\mc{K}_0'$). Then $\widetilde{\sigma}_v$ will denote the
Frobenius automorphism corresponding to $v$ inside $\mc{G}$ and
$\mc{G}_v$ will denote the decomposition group for $v$ in
$\mc{K}/\mc{K}_0'$. Similarly, $\sigma_v$ will denote the
corresponding Frobenius automorphism in $G$ and $G_v$ will denote
the decomposition group for $v$ in $\mc{K}_0/\mc{K}_0'$

If $X$ is any subset of closed points on $Z$, then
$\mathrm{Div}(X)$ will denote the set of divisors on $Z$ which are
supported on $X$. That is $$\mathrm{Div}(X) = \bigoplus_{w\in X}
\Z\cdot w.$$ The degree map $\mathrm{deg}:\mathrm{Div}(X) \ra \Z$
is defined as usual by
$$\mathrm{deg}(\sum_w n_w \cdot w) = \sum_w n_w.$$  The group of
divisors of degree $0$ supported on $X$, $\mathrm{Div}^0(X)$, is
defined by the short exact sequence
$$0 \ra \mathrm{Div}^0(X) \ra \mathrm{Div}(X) \xra{\mathrm{deg}} \Z \ra 0.$$
Finally we define the divisor map
$$\mathrm{div}:\mc{K}^\times \ra \mathrm{Div}^0(Z)$$
by the usual formula
$$\mathrm{div}(f) = \sum_{w\in Z} \mathrm{ord}_w(f) \cdot w.$$

\subsection{Generalized Jacobians and \'{e}tale cohomology groups}

Let $J$ be the Jacobian of $Z$.  $J$ is an abelian variety whose group of $\kappa$-rational points can be identified with the group
$$\mathrm{Pic}^0(Z) := \frac{\mathrm{Div}^0(Z)} {\{\mathrm{div}(f) \mid f \in \mc{K}^\times\}}.$$
As we work exclusively with the $\kappa$-rational points, the letter $J$ will be used to mean the group of $\kappa$-rational points of $J$.

Let $T$ be a finite non-empty set of closed points on $Z$. Define the following subgroup of $\mc{K}^\times$:
$$\mc{K}_T^\times = \{f \in \mc{K}^\times \mid f(v) = 1\textrm{ for all }v\in T\}.$$  The generalized Jacobian $J_T$ is a semi-abelian variety whose group of $\kappa$-rational points can be identified with the group
$$\frac{\mathrm{Div} ^0(Z \setminus T)}{\{\mathrm{div}(f) \mid f \in \mc{K}_T^\times\}}.$$
As above, we will write $J_T$ to mean the group of $\kappa$-rational points of $J_T$.

For every closed point $v$ on $Z$, we let $\kappa(v)$ denote its residue field. Obviously the reduction mod $v$ map leads to a field isomorphism $\kappa\simeq\kappa(v)$, for any such $v$. We let
$$\tau_T := \displaystyle (\bigoplus_{v\in T} \kappa(v)^\times)/\kappa^\times,$$
where $\kappa^\times$ sits in the direct sum above in the usual, diagonal way. The group $\tau_T$ is isomorphic to the group of $\kappa$-rational points of a torus defined over $\kappa_0$.

If $A$ is an abelian group and $\ell$ is a prime, then the $\ell$-adic Tate module of $A$ will be denoted by $T_\ell(A)$.

\begin{proposition}\label{ses of jacobians}
There is a short exact sequence of free $\Z_\ell$-modules of finite rank
$$0\ra T_\ell(\tau_T) \ra T_\ell(J_T) \ra T_\ell(J) \ra 0.$$
\end{proposition}

\begin{proof}
Remark 2.2 in \cite{GP1} shows that there is such an exact sequence so it will suffice to prove that each of these modules is free of finite rank over $\Z_\ell$.  For this it will suffice to prove that there are $\alpha, \beta\in \Z$ such that $\tau_T\tensor \Z_\ell \isom (\Q_\ell/\Z_\ell)^\alpha$ and $J\tensor \Z_\ell\isom (\Q_\ell/\Z_\ell)^\beta$.

First, from the definition of $\tau_T$, it is clear that $\tau_T\tensor \Z_\ell\simeq(\Q_\ell/\Z_\ell)^{(|T|-1)}$ if $\ell\neq p$ and $\tau_T\tensor \Z_p$ is trivial.  Next, it follows from Remark 3.3 in \cite{GP1} that $J\tensor \Z_\ell \isom (\Q_\ell/\Z_\ell)^{2g_Z}$ if $\ell\neq p$ and that $J\tensor \Z_p \isom (\Q_p/\Z_p)^\beta$ for a certain $\beta< g_Z$.  Here $g_Z$ denotes the genus of $Z$. \end{proof}

The following theorem relates the generalized Jacobians associated to $Z$ to the \'{e}tale cohomology groups that we want to study.

\begin{proposition}[Lemma 5.11 and Remark 5.15 in \cite{GP1}]\label{EtaleandJacobian}
Let $S_0$ be a finite set of primes of $\mc{K}_0'$ and let $S$ be the set of primes of $\mc{K}$ lying over $S_0$.  There are isomorphisms of $\Z_\ell[G]$-modules
\begin{enumerate}
\item $H^2_{\acute{e}t}(\mc{O}_{\mc{K}_0, S_0}, \Z_\ell(n)) \xra{\sim} (T_\ell(J_S)(-n)_\Gamma)^\vee$
\item $H^1_{\acute{e}t}(\mc{O}_{\mc{K}_0, S_0}, \Z_\ell(n)) \xra{\sim} (\Q_\ell/\Z_\ell)(n)^{\Gamma} \isom (\Z_\ell(-n)_\Gamma)^\vee$
\end{enumerate}
where $\mc{O}_{\mc{K}_0, S_0}$ is the subring of $\mc{K}_0$ consisting of elements regular away from primes above $S_0$.
\end{proposition}

\begin{proof}
The cited references prove everything but the last isomorphism in $(2)$.  It is easy to see that if $M$ is a $\Z_\ell[[\mc{G}]]$-module, then $(M_\Gamma)^\vee\isom (M^\vee)^\Gamma$.  The relation $\Z_\ell^\vee \isom \Q_\ell/\Z_\ell$ is clear and Remark \ref{twists of duals} implies that $\Z_\ell(-n)^\vee \isom (\Z_\ell^\vee)(n)$.  Putting all this together we therefore have $\Z_\ell[G]$-module isomorphisms
$$(\Z_\ell(-n)_\Gamma)^\vee \isom (\Z_\ell(-n)^\vee)^\Gamma \isom (\Z_\ell^\vee(n))^\Gamma \isom (\Q_\ell/\Z_\ell(n))^\Gamma.$$
\end{proof}

\subsection{The $L$-functions}\label{L-functions}

As in previous sections, let $\mc{K}_0/\mc{K}_0'$ be a Galois extension of characteristic $p$ global fields with Galois group $G$ and let $S_0$ be a finite non-empty set of primes of $\mc{K}_0'$ containing all those primes which ramify in $\mc{K}_0/\mc{K}_0'$.  In addition, let $T_0$ be a finite set of primes of $\mc{K}_0'$ such that $S_0 \cap T_0 \neq \emptyset$.  For each prime $v$ of $\mc{K}_0'$, let $d_v$ denote the residual degree of $v$ over $\kappa_0$. This means that if $\kappa_0(v)$ is the residue field of $v$, then $|\kappa_0(v)| = q^{d_v}$.  To this data we associate the $(S_0, T_0)$ modified equivariant $L$-function, defined by the infinite product
\begin{equation}\label{eulerproduct}\Theta_{\mc{K}_0/\mc{K}_0', S_0, T_0}(u) = \prod_{v\in T_0}(1-\sigma_v^{-1}\cdot (qu)^{d_v}) \cdot \prod_{v\not\in S_0}(1-\sigma_v^{-1}u^{d_v})^{-1},\end{equation}
which is obviously convergent in $Z[G][[u]]$, endowed with the $u$--adic topology.
If $T\neq \emptyset$, then this product converges to an element of $\Z[G][u]$, i.e. it is a polynomial in $u$ with $\Z[G]$--coefficients. (See \S 4.2 of \cite{GP1} and Proposition 2.15 in Chapter 5 of \cite{Tate}.)  To avoid overburdening our notation, we will suppress the extension $\mc{K}_0/\mc{K}_0'$ and simply write $\Theta_{S_0, T_0}(u)$.  If $T_0$ is empty, then we just write $\Theta_{S_0}(u)$.  We note that in the absence of $T_0$, $\Theta_{S_0}(u)$ is not an element of $\Z[G][u]$ in general.



Let $s$ be a complex variable. We define
$$\delta_{T_0}(q^{-s}) := \prod_{v\in T_0} (1-\sigma_v^{-1} (q^{d_v})^{1-s}).$$
Obviously, we have $\Theta_{S_0, T_0}(q^{-s}) = \delta_{T_0}(q^{-s}) \cdot \Theta_{S_0}(q^{-s})$.  This factorization will turn up in later calculations. We remark that if $T_0$ is non-empty, then $$\Theta_{S_0, T_0}(q^{-s}):\C \ra \C[G]$$
is analytic on all of $\C$.  If $T_0 = \emptyset$, then $\Theta_{S_0}(q^{-s})$ will have a pole at $s=1$.
It is however analytic on $\C\setminus \{1\}$ and, in particular, it can be evaluated at negative integers $s=1-n$ for $n\geq 2$.

Fix a prime $\ell \neq p$.  Now, we introduce an element of $\Z_\ell[[\mc{G}]]$ which is the $\mc{G}$-equivariant version of the special values of interest in Theorem \ref{main theorem}.   Recall that $\gamma_q$ denotes the $q$-power Frobenius, which we will now identify with a topological generator for the group $\Gamma' :=G(\mc{K}'/\mc{K}_0')$.  Since $\Theta_{S_0, T_0}(u)\in \Z_\ell[G][u]$, we can evaluate $\Theta_{S_0, T_0}(u)$ at $\gamma_q^{-1}$ to get an element of $\Z_\ell[G \times \Gamma'] \subseteq \Z_\ell[[G\times \Gamma']]$.   Now, via Galois restriction, we can identify $\mathcal G$ with a subgroup
of $G\times\Gamma'$ in the obvious way. In this way, $\Z_\ell[[\mathcal G]]$ is identified with a subring of $\Z_\ell[[G\times \Gamma']]$.

\begin{proposition}
With the above notations, we have $\Theta_{S_0, T_0}(\gamma_q^{-1}) \in \Z_\ell[[\mc{G}]]$.
\end{proposition}

\begin{proof}
See the discussion following Theorem 4.12 in \cite{GP1}.
\end{proof}

 We define $$\vartheta_{S_0, T_0}^{(\infty)} := \Theta_{S_0, T_0}(\gamma_q^{-1}) \in \Z_\ell[[\mc{G}]].$$  The importance of $\vartheta_{S_0, T_0}^{(\infty)}$ is given in the following proposition which shows that its twists know the special values of $\Theta_{S_0, T_0}(q^{-s})$ at negative integers $s=1-n$.

\begin{proposition}\label{Lvalues}
If $\pi:\Z_\ell[[\mc{G}]] \ra \Z_\ell[G]$ is the reduction map, then $$\pi(t_{1-n}(\vartheta_{S_0, T_0}^{(\infty)})) = \Theta_{S_0, T_0}(q^{n-1}).$$
\end{proposition}

\begin{proof}
See the final equation in the proof of Theorem 5.20 in \cite{GP1}
\end{proof}

\section{Picard $1$-Motives}

The concept of a $1$-motive has provided the foundation for recent success in proving classical conjectures on special values of $L$-functions.  Defined by Deligne in \cite{Deligne}, $1$-motives were used by Deligne and Tate to prove the Brumer-Stark conjecture in function fields, see \cite{Tate}.  This  notion of a 1-motive has been generalized by Greither-Popescu in \cite{GP3} through the introduction of their {\em abstract $1$-motives}.  With this machinery they have successfully proven the imprimitive Brumer-Stark conjecture and the Coates-Sinnott conjecture for number fields under certain hypotheses as well as other conjectures on special values of $L$-functions.





\subsection{Galois module structure of Picard $1$-motives}\label{picard}

We keep the setup and notation from the last section.  Let $S$ and $T$ be finite sets of primes of $\mc{K}$ such that $S\cap T = \emptyset$.   The machinery of $1$-motives associates to the data $(\mc{K}, S, T)$ the so-called {\em Picard $1$-motive}
$$\mc{M}_{S, T} = [\mathrm{Div}^0(S) \xra{\delta} J_T]$$
where $\delta$ is the map which sends a divisor to its class in $J_T$.  Note that this makes sense because of the assumption that $S\cap T=\emptyset$.
As in \S2 of \cite{GP1}, for all primes $\ell$ one can define an $\ell$--adic realization ($\ell$--adic Tate module) $T_\ell(\mc{M}_{S, T})$ of $\mc{M}_{S, T}$.
For a given $\ell$, $T_\ell(\mc{M}_{S, T})$ is a free $\Z_\ell$--module of finite rank which sits naturally in an exact sequence  of $\Z_\ell$--modules
\begin{equation}\label{motive jacobian ses}0 \ra T_\ell(J_T) \ra T_\ell(\mc{M}_{S, T}) \ra \mathrm{Div}^0(S)\tensor \Z_\ell\ra 0.\end{equation}
(See loc.cit. for details.)

\begin{remark}
If $S$ and $T$ are $\mc{G}$ invariant, then $T_\ell(\mc{M}_{S, T})$ can be given a natural $\Z_\ell[[\mc{G}]]$-module structure and the maps in the above exact sequence are morphisms of $\Z_\ell[[\mc{G}]]$-modules.
\end{remark}

Greither-Popescu have proven the following theorem on the $\Z_\ell[[\mc{G}]]$-module structure of $T_\ell(\mc{M}_{S, T})$.

\begin{theorem}[Greither-Popescu]\label{GP theorem}
Suppose that $S$ and $T$ are $\mc{G}$ invariant, that $S, T \neq \emptyset$ and that $S$ contains all the primes which ramify in $\mc{K}/\mc{K}'$.  Let $S_0, T_0$ be the sets of primes of $\mc{K}_0'$ which lie below the primes in $S, T$.  In addition, let $H = G(\mc{K}/\mc{K}')$.  Then,
the following hold.
\begin{enumerate}
\item $T_\ell(\mc{M}_{S, T})$ is $\Z_\ell[H]$-projective.
\item $\mathrm{Fit}_{\Z_\ell[[\mc{G}]]}(T_\ell(\mc{M}_{S, T})) = \brkt{\vartheta_{S_0, T_0}^{(\infty)}}$.
\end{enumerate}

\end{theorem}
\begin{proof} See Corollary 4.13 in \cite{GP1}.\end{proof}
The last ingredient needed for the proof of our main theorem is a $\mc{G}$-equivariant duality pairing relating $\mc{M}_{S, T}$ to $\mc{M}_{T, S}$

\begin{proposition}\label{pairing}
Let $\ell\neq p$.  For each $n\in \Z$, there is an isomorphism of $\Z_\ell[[\mc{G}]]$-modules
$$T_\ell(\mc{M}_{S, T})(n-1) \isom T_\ell(\mc{M}_{T, S})(-n)^*$$
\end{proposition}

\begin{proof}
There is a $\Z_\ell[[\mc{G}]]$-equivariant perfect pairing
$$T_\ell(\mc{M}_{S, T}) \times T_\ell(\mc{M}_{T, S}) \ra \Z_\ell(1).$$
(See the proof of Theorem 5.20 in \cite{GP1}.)
This implies that $$T_\ell(\mc{M}_{S, T}) \isom \mathrm{Hom}_{\Z_\ell}(T_\ell(\mc{M}_{T, S}), \Z_\ell(1)) = T_\ell(\mc{M}_{T, S})^*(1).$$
Tensoring with $\Z_\ell(n-1)$, then gives an isomorphism
$$T_\ell(\mc{M}_{S, T})(n-1) \isom T_\ell(\mc{M}_{T, S})^*(n)$$
and applying Remark \ref{twists of duals} finishes the proof.
\end{proof}

\section{Proof of Theorem \ref{main theorem}}

\subsection{Part $(1)$,  $\ell\neq p$}.
We keep the notation of the previous sections but  recall some of it here for the convenience of the reader.  Let $\ell$ be a prime different from $p$ and let $n\in\Z$ with $n\geq 2$.  Let $S_0, T_0$ be two finite sets of primes of $\mc{K}_0'$ such that $S_0$ contains the primes which ramify in $\mc{K}_0/\mc{K}_0'$, $S_0\cap T_0 = \emptyset$ and $S_0, T_0 \neq \emptyset$.  Let $S, T$ denote the primes of $\mc{K}$ lying over the primes in $S_0, T_0$.

We begin by writing exact sequence \eqref{motive jacobian ses}, with $S$ and $T$ swapped:
$$ 0 \ra T_\ell(J_S) \ra T_\ell(\mc{M}_{T, S}) \ra \mathrm{Div}^0(T)\tensor \Z_\ell \ra 0.$$
Now, the definition of $\mathrm{Div}^0(T)$  combined with the exact twisting functor leads to an exact sequence of $\Z_\ell[[\mathcal G]]$--modules
\begin{equation}\label{four term sequence} 0 \ra T_\ell(J_S)(-n) \ra T_\ell(\mc{M}_{T, S})(-n) \xra{\vp} \mathrm{Div}(T)\tensor \Z_\ell(-n) \ra \Z_\ell(-n) \ra 0.\end{equation}

\begin{proposition}
Upon taking $\Gamma$ coinvariants, sequence \eqref{four term sequence} induces an exact sequence of $\Z_\ell[G]$-modules
\begin{equation}\label{key sequence}0 \ra T_\ell(J_S)(-n)_\Gamma \ra T_\ell(\mc{M}_{T, S})(-n)_\Gamma \ra (\mathrm{Div}(T)\tensor \Z_\ell(-n))_\Gamma \ra Z_\ell(-n)_\Gamma \ra 0.\end{equation}  Furthermore, all of these modules are finite.
\end{proposition}

\begin{proof}
We can break \eqref{four term sequence} into the two short exact sequences
$$0 \ra T_\ell(J_S)(-n) \ra T_\ell(\mc{M}_{T, S})(-n) \ra \ker \vp \ra 0$$
and $$0 \ra \ker \vp \ra \mathrm{Div}(T)\tensor \Z_\ell(-n) \ra \Z_\ell(-n) \ra 0.$$
To check the exactness claimed in the proposition, it will suffice to check that each of these sequences stays exact when we take coinvariants.  Lemma \ref{snake} implies that there are exact sequences
$$(\ker \vp)^\Gamma \ra T_\ell(J_S)(-n)_\Gamma \ra T_\ell(\mc{M}_{T, S})(-n)_\Gamma \ra (\ker \vp)_\Gamma \ra 0$$
and
$$Z_\ell(-n)^\Gamma \ra (\ker \vp)_\Gamma \ra (\mathrm{Div}(T) \tensor Z_\ell(-n))_\Gamma \ra \Z_\ell(-n)_\Gamma \ra 0.$$
So, we must show that $(\ker\vp)^\Gamma = 0$ and $(\Z_\ell(-n))^\Gamma =0$.  As a free $\Z_\ell$-module has no non-zero finite submodules, it will suffice to show that $(\ker\vp)^\Gamma$ and  $(\Z_\ell(-n))^\Gamma$ are both finite.  By Lemma \ref{finiteness of invariants} this is equivalent to showing that $(\ker\vp)_\Gamma$ and $\Z_\ell(-n)_\Gamma$ are both finite and this is what we will prove.

Let $\alpha \in \N$, such that $\kappa\cap\mathcal K_0=\Bbb F_{q^\alpha}$. Then $\gamma :=\gamma_q^\alpha$ is a topological generator for $\Gamma$. (Note that $\Gamma$ can be identified via Galois restriction with an open subgroup of $\Gamma'$.)  There is an obvious isomorphism of $\Z_\ell[[\Gamma]]$-modules $\Z_\ell(-n) \isom \Z_\ell[[\Gamma]]/\brkt{1-c_\ell(\gamma)^n \gamma}$.   Taking $\Gamma$ coinvariants we have an isomorphism
$$\Z_\ell(-n)_\Gamma \isom \Z_\ell[[\Gamma]]/\brkt{1-c_\ell(\gamma)^n\gamma, 1-\gamma} \isom \Z_\ell/\brkt{1-c_\ell(\gamma)^n}.$$
Since $c_\ell(\gamma)^n \neq 1$ for any $n\geq 1$, this is a quotient of two free $\Z_\ell$-modules of rank $1$ and therefore it is finite.

Now, since $(\ker\vp)_\Gamma \subseteq (\mathrm{Div}(T)\tensor \Z_\ell(-n))_\Gamma$, it will suffice to show that the latter group is finite.  We can write
\begin{equation}\label{divT}\mathrm{Div}(T)\tensor \Z_\ell(-n)\simeq \bigoplus_{v \in T_0}(\bigoplus_{w|v} \Z_\ell\cdot w)(-n)
\end{equation}
where for each $v\in T_0$ $w$ runs through the set of all primes $w$ in $T$ sitting over $v$. Note that each inner sum is a $\Z_\ell[[\Gamma]]$-module.
We will prove that $(\bigoplus_{w|v} \Z_\ell \cdot w)(-n)_\Gamma$ is finite, for all $v\in T_0$.

Fix a prime $v\in T_0$ and let $w_0$ be a choice of a prime of $\mc{K}$ lying over $v$.  We know that $\mc{G}$ acts transitively on the primes of $\mc{K}$  lying above $v$ and that $\mc{G}_v$ is the stabilizer of $w_0$ under this action.  Consequently, we have isomorphisms of $\Z_\ell[[\mathcal G]]$--modules
$$(\bigoplus_{w|v} \Z_\ell \cdot w)(-n)  \simeq (\Z_\ell[[\mathcal G]]/\brkt{1-\widetilde\sigma_v})(-n)\simeq \Z_\ell[[\mc{G}]]/\brkt{1-c_\ell(\widetilde{\sigma}_v)^n \widetilde{\sigma}_v}.$$
For the second isomorphism, combine Proposition \ref{properties of fitting ideals} part (2) and Proposition \ref{linear algebra}. It follows that we have an isomorphism of $\Z_\ell[G]$--modules
\begin{equation}\label{coinv-divT}
(\bigoplus_{w|v} \Z_\ell\cdot w)(-n)_\Gamma \isom \Z_\ell[G]/\brkt{1-c_\ell(\widetilde{\sigma}_v)^n\sigma_v}.
\end{equation}
Now, we obviously have $c_\ell(\widetilde{\sigma}_v)=q^{d_v}$. Consequently, since $n\ne 0$, the element $(1-c_\ell(\widetilde{\sigma}_v)^n \sigma_v)=(1-q^{n\cdot d_v}\cdot\sigma_v)$ is not a zero-divisor in $\Z_\ell[G]$. In order to see this, simply note that for all $\overline {\Bbb Q}_\ell$--valued characters $\chi$ of $G$, we have
$$\chi(1-q^{n\cdot d_v}\cdot\sigma_v)= 1-q^{n\cdot d_v}\cdot\chi(\sigma_v)\ne 0$$
and recall that the zero-divisors of $\Z_\ell[G]$ are those elements which lie in the kernel of at least one character $\chi$. In fact, these kernels are precisely the minimal prime
ideals of $\Z_\ell[G]$. It is now clear that the module on the right side of the last displayed isomorphism is finite, as it is a quotient of two finitely generated free $\Z_\ell$--modules of equal rank.

This proves the exactness of \eqref{key sequence} and along the way we have also managed to prove that two of the four modules appearing in this sequence, namely the last two, are finite.  If we show that either one of the two remaining modules is finite, then the finiteness of the other will follow. We will prove that $T_\ell(J_S)(-n)_\Gamma$ is finite.

Using Lemma \ref{finiteness of invariants} it is easy to see that $T_\ell(J_S)(-n)_\Gamma$ is finite if and only if $1$ is not an eigenvalue for the action of $\gamma$ on $T_\ell(J_S)(-n)\tensor \Q_\ell$.  From Proposition \ref{ses of jacobians}, we have a short exact sequence $$0\ra T_\ell(\tau_S)(-n) \ra T_\ell(J_S)(-n) \ra T_\ell(J) (-n)\ra 0.$$
Now, recall that the Riemann hypothesis for the curve $Z$ (a theorem due to A. Weil in this context) says that the action of $\gamma$ on $T_\ell(J)\tensor_{\Z_\ell} \Q_\ell$ has eigenvalues which are algebraic integers independent of $\ell$ and of absolute value $(q^\alpha)^{1/2}$.  It follows that the eigenvalues for the action of $\gamma$ on $T_\ell(J)(-n) \tensor_{\Z_\ell}\Q_\ell$ all have absolute value $(q^\alpha)^{1/2-n}$. (See Remark \ref{eigenvalues of twists}.) Since $n$ is an integer, $1$ is not an eigenvalue of this action.  Similarly, all the eigenvalues of $\gamma$ on $T_\ell(\tau_S)(-n)$ are equal to $(q^\alpha)^{(1-n)}$.  Since $n>1$, we therefore see that $1$ is not an eigenvalue of the action of $\gamma$ on $T_\ell(J_S)(-n)$ and this finishes the proof.
\end{proof}

We intend to apply Proposition \ref{four term lemma} to the sequence obtained by dualizing sequence \eqref{key sequence}.  It will be convenient to check that sequence \eqref{key sequence} satisfies the hypotheses of Proposition \ref{four term lemma} and then dualize.  Proposition \ref{projective dimension argument} ensures that Proposition \ref{four term lemma} will still apply to the dualized sequence.  The only hypotheses that remain to be checked are that the two modules appearing in the middle of the sequence have projective dimension at most $1$.

\begin{proposition}
Both $(\mathrm{Div}(T)\tensor \Z_\ell(-n))_\Gamma$ and $T_\ell(\mc{M}_{T, S})(-n)_\Gamma$ have projective dimension equal to 1 over $\Z_\ell[G]$.
\end{proposition}

\begin{proof}
Theorem \ref{GP theorem} implies that $T_\ell(\mathcal M_{S, T})$ is $\Z_\ell[H]$-projective for $H=G(\mc{K}/\mc{K}')$.  Proposition 1 of \S3, Chapter 9 in \cite{Local-Fields} then implies that $T_\ell(\mathcal M_{S, T})(n-1)^*$ is $\Z_\ell[H]$--projective as well.  Proposition \ref{pairing} now implies that $T_\ell(\mathcal M_{T, S})(-n)$ is also $\Z_\ell[H]$--projective.

Let $\mc{K}_\infty, \mc{K}_\infty'$ denote the cyclotomic $\Z_\ell$-extensions of $\mc{K}_0$ and $\mc{K}_0'$, respectively. By definition, $\mc{K}_\infty$ and $\mc{K}_\infty'$ are the maximal pro-$\ell$ subextensions of $\mc{K}$ and $\mc{K'}$, respectively and the Galois groups $\Gamma_\ell:=G(\mc{K}_\infty/\mc{K}_0)$ and $\Gamma_\ell' := G(\mc{K}_\infty'/\mc{K}_0')$ are isomorphic to $\Z_\ell$.  Now, $\mathcal K$ is the compositum of the linearly disjoint cyclotomic $\Z_r$--extensions of $\mathcal K_0$, for all primes $r$. Consequently, we can write $\Gamma \isom \widehat{\Z} \isom \prod_r \Z_r$, where $r$ runs over all the primes.  As $T_\ell(\mathcal M_{T,S})$ is a free $\Z_\ell$-module, say of rank $m$, the action of $\Gamma$ on $T_\ell(\mathcal M_{T, S})$ factors through ${\rm Aut}_{\Z_\ell}\,T_\ell(\mathcal M_{T, S})\simeq GL_m(\Z_\ell)$.
It is easy to see that the prime-to-$\ell$ part of $GL_m(\Z_\ell)$ is finite, which implies that the image of $\prod_{r\neq \ell} \Z_r \subseteq \widehat{\Z}$ in $GL_m(\Z_\ell)$ is finite.

This shows that the action of $\Gamma$ actually factors through the Galois group of a subextension of $\mc{K}/\mc{K}_0'$ which contains and is finite over $\mc{K}_\infty$. In other words, for $N$ sufficiently large, $T_\ell(\mathcal M_{T, S})$ is a module over $\Z_\ell[[G(\mc{K}_\infty(\mu_N)/\mc{K}_0')]]$. We fix such an $N$ with the additional property that $\ell\mid N$.  Then, $\mu_{\ell^\infty} \subseteq \mc{K}_\infty(\mu_N)$ and therefore $T_\ell(\mathcal M_{T, S})(-n)$ is a $\Z_\ell[[G(\mc{K}_\infty(\mu_N)/\mc{K}_0')]]$-module.  Let $\Gamma_{\ell, N}:= G(\mc{K}_\infty(\mu_N)/\mc{K}_0)$. Then, we obviously have equalities
$$T_\ell(\mathcal M_{T, S})(-n)_\Gamma = T_\ell(\mathcal M_{T, S})(-n)_{\Gamma_{\ell, N}}, \qquad T_\ell(\mathcal M_{T, S})(-n)^\Gamma = T_\ell(\mathcal M_{T, S})(-n)^{\Gamma_{\ell, N}}.$$

Recall that we have defined $H=G(\mc{K}/\mc{K}')$.  Let us write $H' = G(\mc{K}_\infty(\mu_N)/\mc{K}_\infty')$.  It is easy to see that $H$ is isomorphic to a subgroup of $H'$ and that the quotient $H'/H$ has order co-prime to $\ell$.  There is an exact sequence
$$0\ra H' \ra G(\mc{K}_\infty(\mu_N)/\mc{K}_0') \ra \Gamma_\ell' \ra 0$$
and the fact that $\Gamma_\ell'$ is topologically cyclic implies that this sequence is split.  We can therefore write
$$\Z_\ell[[G(\mc{K}_\infty(\mu_N)/\mc{K}_0')]] \isom\Z_\ell[[\Gamma_\ell']][H'].$$
Now, Proposition 2.2 and Lemma 2.3 in \cite{CS} lead to the following.

\begin{proposition} With notations as above, let $\Lambda:=\Z_\ell[[\Gamma_\ell']]$ and let $M$ be a finitely generated $\Lambda[H']$--module. Then, the following are equivalent.
\begin{enumerate}
\item ${\rm pd}_{\Lambda[H']}M=1$.
\item $M$ has no non-trivial finite $\Lambda$--submodules and $M$ is $H$--cohomologically trivial.
\end{enumerate}
\end{proposition}

\begin{proof} See Proposition 2.2 and Lemma 2.3 in \cite{CS}. \end{proof}

We apply the above Proposition to the $\Lambda[H']$--module $T_\ell(\mathcal M_{T, S})(-n)$.  Since $T_\ell(\mathcal M_{T, S})(-n)$ is $H$-cohomologically trivial (see Theorem \ref{GP theorem}(1)), the fact that $\ell \nmid |H'/H|$ combined with a standard inflation--restriction argument implies that $T_\ell(\mathcal M_{T, S})(-n)$ is $H'$--cohomologically trivial as well.  It is clear that $T_\ell(M_{T, S})(-n)$ has no finite $\Lambda$ submodules (as it is $\Z_\ell$--free). The above Proposition implies that there is a short exact sequence of $\Lambda[H']$--modules
$$ 0 \ra P_1 \ra P_0 \ra T_\ell(\mathcal M_{T, S})(-n) \ra 0$$
where $P_1, P_0$ are projective over $\Lambda[H']$.

The isomorphism $G(\mc{K}_\infty(\mu_N)/\mc{K}_0')/\Gamma_{\ell, N}\isom G$ is clear and so Lemma \ref{snake} implies that there is an exact sequence of $\Z_\ell[G]$-modules
\begin{equation}\label{sequence-proj}
T_\ell(\mathcal M_{T, S})(-n)^{\Gamma_{\ell, N}} \ra (P_1)_{\Gamma_{\ell, N}} \ra (P_0)_{\Gamma_{\ell, N}} \ra T_\ell(\mathcal M_{T, S})(-n)_{\Gamma_{\ell, N}}\ra 0.
\end{equation}

Now, Lemma \ref{finiteness of invariants} implies that $T_\ell(\mathcal M_{T, S})(-n)^{\Gamma_{\ell, N}}=T_\ell(\mathcal M_{T, S})(-n)^{\Gamma}=0$. Indeed,
we have already shown that $T_\ell(\mathcal M_{T, S})(-n)_{\Gamma}$ is finite, which implies that $T_\ell(\mathcal M_{T, S})(-n)^{\Gamma}$ is finite. However, since $T_\ell(\mathcal M_{T, S})(-n)$ has no finite $\Z_\ell$-submodules,
$T_\ell(\mathcal M_{T, S})(-n)^{\Gamma}$ must be trivial.

Since the modules $(P_i)_{\Gamma_{\ell, N}}$ are clearly projective over $\Z_\ell[G]$, exact sequence \eqref{sequence-proj} implies that $\mathrm{pd}_{\Z_\ell[G]}(T_\ell(\mathcal M_{T, S})(-n)_\Gamma)=1$,
as required.

 Next, we deal with $(\mathrm{Div}(T)\tensor \Z_\ell(-n))_\Gamma$.  In the previous proposition, we showed that $(\mathrm{Div}(T)\tensor \Z_\ell(-n))_\Gamma$ is isomorphic as a $\Z_\ell[G]$-module to a direct sum of modules of the form $\Z_\ell[G]/\brkt{1-c_\ell(\widetilde{\sigma}_v)^n \sigma_v}$.  As $1-c_\ell(\widetilde{\sigma}_v)^n \sigma_v$ is not a zero-divisor in $\Z_\ell[G]$, this implies that we have an exact sequence of $\Z_\ell[G]$-modules
$$0 \ra \Z_\ell[G] \xra{1-c_\ell(\widetilde{\sigma}_v)^n\sigma_v} \Z_\ell[G] \ra \Z_\ell[G]/\brkt{1-c_\ell(\widetilde{\sigma}_v)^n \sigma_v} \ra 0.$$
This shows that $\Z_\ell[G]/\brkt{1-c_\ell(\widetilde{\sigma}_v)^n \sigma_v}$ has projective dimension $1$ over $\Z_\ell[G]$ and therefore the same is true of $\mathrm{Div}(T) \tensor Z_\ell(-n)_\Gamma$.  This concludes the proof.
\end{proof}

Applying $\mathrm{Hom}_{\Z_\ell}(-, \Q_\ell/\Z_\ell)$ to \eqref{key sequence} we arrive at the exact sequence
\begin{eqnarray}\label{sequence of duals}
0\longrightarrow (\Z_\ell(-n)_\Gamma)^\vee \longrightarrow ((\mathrm{Div}(T)\tensor \Z_\ell(-n))_\Gamma)^\vee \longrightarrow\\
\nonumber \qquad \longrightarrow  (T_\ell(\mc{M}_{T, S})(-n)_\Gamma)^\vee \longrightarrow (T_\ell(J_S)(-n)_\Gamma)^\vee \longrightarrow 0.\end{eqnarray}
We have observed that this sequence satisfies the hypotheses of Proposition \ref{four term lemma}.

In the remainder of this section we adopt the convention that all Fitting ideals are taken over $\Z_\ell[G]$ unless specified.  Applying Proposition \ref{four term lemma} to sequence \eqref{sequence of duals} we have an equality of ideals
\begin{eqnarray}\label{fundamentalEquality}
 \mathrm{Fit}(((\Z_\ell(-n)_\Gamma)^\vee)^\wedge)\cdot  \mathrm{Fit}((T_\ell(\mc{M}_{T, S})(-n)_\Gamma)^\vee)  &=&  \\
\nonumber  =\mathrm{Fit}(((\mathrm{Div}(T)\tensor\Z_\ell(-n))_\Gamma)^\vee) \cdot \mathrm{Fit}((T_\ell(J_S)(-n)_\Gamma)^\vee) &&  
\end{eqnarray}
What remains is to calculate each of these Fitting ideals individually and fit all the pieces together.
By combining \eqref{divT} and \eqref{coinv-divT} we get a $\Z_\ell[G]$--module isomorphism
$$\mathrm{Div}(T)\tensor \Z_\ell(-n) \isom \bigoplus_{v\in T_0} \Z_\ell[[\mc{G}]]/\brkt{1-c_\ell(\widetilde{\sigma}_v)^n \sigma_v}.$$
Now, Proposition \ref{properties of fitting ideals} parts $(2)$ and $(3)$ imply that
$$\mathrm{Fit}(\mathrm{Div}(T)\tensor\Z_\ell(-n))_\Gamma) = \brkt{\prod_{v\in T_0}(1-c_\ell(\widetilde{\sigma}_v)^n\cdot\sigma_v)} = \brkt{\prod_{v\in T_0}(1-Nv^n\cdot\sigma_v)}.$$
If we now apply Proposition \ref{projective dimension argument} we see that we have an equality of ideals
\begin{eqnarray}
\nonumber \mathrm{Fit}((\mathrm{Div}(T)\tensor\Z_\ell(-n))_\Gamma^\vee) = \brkt{\iota(\prod_{v\in T_0}(1-Nv^n\cdot\sigma_v))}=\\
\nonumber = \brkt{\prod_{v\in T_0}(1-Nv^n\cdot\sigma_v^{-1})} = \brkt{\delta_{T_0}(q^{n-1})}.
\end{eqnarray}
Applying Lemma \ref{duality}, and Proposition \ref{pairing} we get that
$$(T_\ell(\mc{M}_{T, S})(-n)_\Gamma)^\vee \isom (T_\ell(\mc{M}_{T, S})(-n)^*)_\Gamma \isom T_\ell(\mc{M}_{S,T})(n-1)_\Gamma.$$
Combining Theorem \ref{GP theorem} with Corollary \ref{linear algebra} we have that $$\mathrm{Fit}_{\Z_\ell[[\mc{G}]]}(T_\ell(\mc{M}_{S, T})(n-1)) = \brkt{t_{1-n}(\vartheta_{S_0, T_0}^{(\infty)})}.$$
Applying Proposition \ref{Lvalues}, we then have that
$$\mathrm{Fit}(T_\ell(\mc{M}_{S, T})(n-1)_\Gamma) = \brkt{\pi(t_{1-n}(\vartheta_{S_0, T_0}^{(\infty)}))} = \brkt{\delta_{T_0}(q^{n-1})\cdot \Theta_{S_0}(q^{n-1})}.$$
The \'{e}tale cohomology groups appear via Proposition \ref{EtaleandJacobian} which says that
$$(\Z_\ell(-n)_\Gamma)^\vee\isom \mathrm{H}^1_{\acute{e}t}(\mc{O}_{\mc{K}_0, S_0}, \Z_\ell(n)) \textnormal{ and } (T_\ell(J_S)(-n)_\Gamma)^\vee \isom \mathrm{H}^2_{\acute{e}t}(\mc{O}_{\mc{K}_0,S_0}, \Z_\ell(n)).$$
Finally, Lemma \ref{fitting and annhilators} implies that $$\mathrm{Fit}(((\Z_\ell(-n)_\Gamma)^\vee)^\wedge) = \mathrm{Fit}((\Z_\ell(-n)_\Gamma)^\vee).$$
Combining all these calculations we can rewrite \eqref{fundamentalEquality} as
$$\mathrm{Fit}(\mathrm{H}^1_{\acute{e}t}(\mc{O}_{\mc{K}_0, S_0}, \Z_\ell(n))) \cdot \delta_{T_0}(q^{n-1})\cdot \Theta_{S_0}(q^{n-1})=\hspace{2in}$$
$$\hspace{1.5in}=\delta_{T_0}(q^{n-1}) \cdot  \mathrm{Fit}(\mathrm{H}^2_{\acute{e}t}(\mc{O}_{\mc{K}_0, S_0}, \Z_\ell(n))).$$
We have seen that $\delta_{T_0}(q^{n-1})$ generates the Fitting ideal of a torsion finitely generated $\Z_\ell[G]$-module with projective dimension 1, and so Proposition \ref{projective dimension argument} implies that it is not a zero-divisor in $\Z_\ell[G]$.  We are therefore entitled to cancel the $\delta_{T_0}(q^{n-1})$ term from both sides of the equation.  This results in the equality
$$\mathrm{Fit}_{\Z_\ell[G]}(\mathrm{H}^1_{\acute{e}t}(\mc{O}_{\mc{K}_0, S_0}, \Z_\ell(n)))\cdot \Theta_{S_0}(q^{n-1})=\mathrm{Fit}_{\Z_\ell[G]}(\mathrm{H}^2_{\acute{e}t}(\mc{O}_{\mc{K}_0, S_0}, \Z_\ell(n)))$$
and this concludes the proof of part (1) of the theorem.

\hfill  $\square$

\subsection{Part (2), $\ell=p$} Recall that in this case we have to prove that
\begin{equation}\label{goal-p}\Theta_{S_0}(q^{n-1})\in\Z_p[G]^\times,\qquad \text{for all $n\in\Z_{\geq 2}$,}\end{equation}
where the notations are as above. We need the following.

\begin{lemma}
We have an inclusion $$1+p\Z_p[G]\subseteq \Z_p[G]^\times.$$
\end{lemma}
\begin{proof}
Since $\Z_p[G]$ is an integral extension of $\Z_p$ and $p\Z_p$ is the only maximal ideal in $\Z_p$,
the going-up theorem of Cohen-Seidenberg
implies that $p$ belongs to the Jacobson radical of $\Z_p[G]$. Consequently, the inclusion in
the statement of the Lemma holds.
\end{proof}
Now, we fix a finite, non-empty set $T_0$ of closed points of $Z_0'$, such that $S_0$ and $T_0$ are disjoint.
\begin{corollary}
For all $n\in\Z_{\geq 2}$, we have
$$\delta_{T_0}(q^{n-1})\in\Z_p[G]^\times.$$
\end{corollary}
\begin{proof} Since $q^{n-1}$ is a power of $p$, we have
$$\delta_{T_0}(q^{n-1})=\prod_{v\in T_0}(1- q^{n\cdot d_v}\cdot \sigma_v^{-1})\in 1+p\Z_p[G]$$
and the above Lemma implies the desired result.\end{proof}

Since we already know that $\Theta_{S_0,T_0}(u)\in\Z[G][u]$, we have
$$\delta_{T_0}(q^{n-1})\cdot\Theta_{S_0}(q^{n-1})=\Theta_{S_0, T_0}(q^{n-1})\in\Z_p[G].$$
Consequently, the Corollary above implies that \eqref{goal-p} is equivalent to
\begin{equation}\label{newgoal-p}\Theta_{S_0, T_0}(q^{n-1})\in\Z_p[G]^\times,\qquad \text{for all $n\in\Z_{\geq 2}$.}\end{equation}
In order to prove this, we need to recall a few facts from \cite{GP1}. First, recall that if $P$ is a finitely generated projective module over a commutative ring $R$,
$f\in{\rm End}_R(P)$ and $u$ is a variable, then the polynomial
$${\rm det}_R(1-f\cdot u\mid P)\in R[u]$$
is well defined and it is equal to
$${\rm det}_{R[u]}(1_{P\oplus Q}\otimes 1-(f\oplus 1_Q)\otimes u\mid (P\oplus Q)\otimes_R R[u]),$$
where $Q$ is any finitely generated $R$--module such that $(P\oplus Q)$ is $R$--free and $1_Q$  and $1_{P\oplus Q}$ are the identity
endomorphisms of $Q$ and $P\oplus Q$, respectively. Consequently, we have
\begin{equation}\label{projective-general}{\rm det}_R(1-f\cdot u\mid P)\in 1+u\cdot R[u],\end{equation}
for all $R$, $P$ and $f$ as above.

Next, let $\overline Z:=Z_0\times_{\kappa_0}\kappa$ be the smooth, projective (not necessarily connected) curve
over ${\rm Spec}(\kappa)$ obtained by base change from $Z_0$. Also, let $\overline S$ and $\overline T$ be the (finite, disjoint)
sets of closed points on $\overline Z$ lying above points in $S_0$ and $T_0$, respectively. We denote by $\overline{\mathcal M}$ the Picard $1$--motive
associated to the data $(\overline Z, \overline S, \overline T)$. By definition, $\overline{\mathcal M}$ is defined over $\kappa_0$ and $G$ acts
naturally on $\overline{\mathcal M}$ via $\kappa_0$--automorphisms. Consequently, the $\ell$--adic realizations $T_{\ell}(\overline{\mathcal M})$
of $\overline{\mathcal M}$ are endowed with natural $\Z_\ell[G]$--module structures, for all prime numbers $\ell$.
The following was proved in \cite{GP1}.

\begin{theorem}[Greither-Popescu] The following hold.

\begin{enumerate}
\item For all primes $\ell$, $T_\ell(\overline{\mathcal M})$ is a projective $\Z_l[G]$--module.
\item For all primes $\ell\ne p$, we have
$$\Theta_{S_0, T_0}(u)={\rm det}_{\Z_{\ell}[G]}(1-\gamma_q\cdot u\mid T_\ell(\overline{\mathcal M})).$$
\end{enumerate}
\end{theorem}

\begin{remark} Statement (2) in the Theorem above can be extended to the prime $p$ as well, provided that one replaces the $p$--adic realization
$T_p(\overline{\mathcal M})$ with the crystalline realization
$T_{\rm crys}(\overline{\mathcal M})$ of $\overline{\mathcal M}$,
as defined in \cite{Andreatta-Popescu}. The module $T_{\rm
crys}(\overline{\mathcal M})$ is finitely generated over $W[G]$,
where $W$ is the ring of Witt vectors of $\kappa_0$. Although in
general $T_{\rm crys}(\overline{\mathcal M})$ is not projective
over $W[G]$ (see \cite{Andreatta-Popescu}), one still has
 $$\Theta_{S_0, T_0}(u)={\rm det}_{\C_p[G]}(1-\gamma_q\cdot u\mid T_{\rm crys}(\overline{\mathcal M})\otimes_W\C_p),$$
where $\C_p$ is the completion of the algebraic closure of $\Q_p$.
\end{remark}

Now, part (2) of the Theorem above (for a fixed prime $\ell\ne p$)
combined with \eqref{projective-general} and the fact that
$\Theta_{S_0, T_0}(u)\in\Z[G][u]$ imply that we have
$$\Theta_{S_0,T_0}(u)\in (1+u\cdot\Z_\ell[G][u])\cap \Z[G][u]=1+u\Z[G][u].$$
Consequently, since $q$ is a power of $p$, we have
$$\Theta_{S_0,T_0}(q^{n-1})\in 1+p\cdot\Z_p[G], \qquad\text{for all $n\in\Z_{\geq 2}$.}$$
Now, one applies the Lemma above to conclude that
\eqref{newgoal-p} holds. This concludes the proof of Theorem
\ref{main theorem}. \hfill$\square$
\medskip

As proved in the introduction, the following refinement of the characteristic $p$ analogue of the Coates--Sinnott conjecture
is a consequence of Theorem \ref{main theorem}.
\begin{corollary}
Let $n\geq 2$ be an integer.  Then
$$\mathrm{Fit}_{\Z[G]}(K_{2n-1}(\mc{O}_{\mc{K}_0, S_0})) \cdot \Theta_{\mc{K}_0/\mc{K}_0', S_0}(q^{n-1}) = \mathrm{Fit}_{\Z[G]}(K_{2n-2}(\mc{O}_{\mc{K}_0, S_0})).$$
\end{corollary}
\begin{proof}
See the proof of Theorem \ref{refined CS} in the Introduction.
\end{proof}

\bibliographystyle{plain}
\bibliography{Dodge-PopescuBib}

\end{document}